\documentclass[oneside,leqno,11pt]{paper}
\usepackage{amsmath,amssymb,latexsym}
\usepackage{pdfpages}
\usepackage[colorinlistoftodos]{todonotes}
\usepackage{cite}
\usepackage{multirow}
\usepackage{pgfplots}

\pgfplotsset{compat=1.7}

\usepackage{graphicx}
\usepackage[english]{babel}
\usepackage[utf8]{inputenc}
\usepackage{amsfonts}
\usepackage{hyperref}
 	 	\usepackage{enumerate}
\usepackage{bibentry}
\nobibliography*

\usepackage{graphicx}
\usepackage{float}

\makeatletter
\def\ball{{I\kern -.35em B}}
\def\tto{\rightrightarrows}
\def\bx{\bar x}
\def\by{\bar y}
\def\bz{\bar z}

\def\nn{neighborhood\,}
\def\nns{neighborhoods\,}

\def\dom{\mathop{\rm dom}\nolimits}

\def\gph{\mathop{\rm gph}\nolimits}

\def\dist{\mathop{\rm dist}\nolimits}

\newtheorem{proof}{Proof}

\newtheorem{definition}{Definition}[section]
\newtheorem{proposition}{Proposition}[section]
\newtheorem{theorem}{Theorem}[section]

\newtheorem{remark}{Remark}[section]
\newtheorem{example1}{Example}[section]

\pgfdeclarepatternformonly[\LineSpace]{my north east lines}{\pgfqpoint{-1pt}{-1pt}}{\pgfqpoint{\LineSpace}{\LineSpace}}{\pgfqpoint{\LineSpace}{\LineSpace}}%
{
	\pgfsetcolor{\tikz@pattern@color}
	\pgfsetlinewidth{0.1pt}
	\pgfpathmoveto{\pgfqpoint{0pt}{0pt}}
	\pgfpathlineto{\pgfqpoint{\LineSpace + 0.1pt}{\LineSpace + 0.1pt}}
	\pgfusepath{stroke}
}

\pgfdeclarepatternformonly[\LineSpace]{my north west lines}{\pgfqpoint{-1pt}{-1pt}}{\pgfqpoint{\LineSpace}{\LineSpace}}{\pgfqpoint{\LineSpace}{\LineSpace}}%
{
	\pgfsetcolor{\tikz@pattern@color}
	\pgfsetlinewidth{0.1pt}
	\pgfpathmoveto{\pgfqpoint{0pt}{\LineSpace}}
	\pgfpathlineto{\pgfqpoint{\LineSpace + 0.1pt}{-0.1pt}}
	\pgfusepath{stroke}
}
\makeatother

\newdimen\LineSpace
\tikzset{
	line space/.code={\LineSpace=#1},
	line space=5pt
}
\begin{document}

\title{A note on the uniformity of strong subregularity around the reference point}
\author{Tomáš Roubal\thanks{Institute of Information Theory and Automation, Czech Academy of Sciences, Prague, Czech Republic, 	 roubal@utia.cas.cz, ORCID ID: 0000-0002-6137-1046}}

\date{}
\maketitle

\begin{abstract}
This paper investigates strong metric subregularity around a reference point as introduced by H. Gfrerer and J. V. Outrata in \cite{GO2022}. In the setting of Banach spaces, we analyse its stability under Lipschitz continuous perturbations and establish its uniformity over compact sets. Our results ensure that the property is preserved under small Lipschitz perturbations, which is crucial for maintaining robustness in variational analysis. Furthermore, we apply the developed theory to parametric inclusion problems. The analysis demonstrates that the uniformity of strong metric subregularity provides a  theoretical foundation for addressing stability issues in parametrized optimization and control applications.

\end{abstract}


\section{Introduction}

Stability and robustness of solutions in optimization and control are central issues in variational analysis, where understanding local solution behavior and its sensitivity to perturbations is essential. Recently, inspired by the Newton method, H.  Gfrerer and  J. V. Outrata \cite{GO2021} introduced the semismooth$^*$ method, bringing renewed attention to the concept of strong metric subregularity, e.g., \cite{CDK2018}.
The property was extended by the same authors to strong metric subregularity around the reference, which enables the analysis of solution stability and the convergence the method, see \cite{BM2021}.

The aim of this study is to demonstrate that strong metric subregularity is preserved and even becomes uniform using Lipschitz continuous perturbations. Building upon these foundations, we prove the existence of uniform constants and domains that ensure the subregularity property holds uniformly across compact sets. Such uniformity plays a crucial role in the robustness of numerical methods and path-following techniques for solving parametric inclusion problems, e.g., \cite{DCH2013, CR2018, DKRV2013}.

In the sections that follow, we first establish the basic notation and necessary definitions within the framework of Banach spaces. We then present the main theoretical results, demonstrating that local strong metric subregularity can be extended to uniform subregularity on compact sets. This extension not only extends the theoretical foundations of variational analysis but also opens up new avenues for the construction and analysis of efficient algorithms in optimization, control, and economic modelling.

 \section{Preliminaries}
	Throughout the whole paper, we assume that $X$ and $Y$ are Banach spaces and  $P$ is a metric space. The closed ball and open ball of radius $\delta$ centred at a point $x \in X$ are defined respectively as
 $$
\ball_X[x, \delta] := \{u \in X : \Vert x- u\Vert \leq \delta\}
\quad \text{and}\quad
 \ball_X(x, \delta) = \{u \in X : \Vert x- u\Vert < \delta\}.
 $$

 The distance from a point $x \in X$ to a set ${A}$ is denoted by $\dist(x, {A})$ and is defined as the shortest distance between $x$ and any point in ${A}$, expressed as $\dist(x, {A}) = \inf_{u \in {A}}\Vert u- x\Vert.$

 The graph of a set-valued mapping $F$, represented as $\gph\, F$, comprises all pairs $(x, y)$ such that $y \in F(x)$. Additionally, the domain of $F$, denoted by $\dom\, F$, includes all points $x$ for which the set $F(x)$ is nonempty, indicating the extent of the definition of $F$. The inverse of a set-valued mapping $F$, denoted by $F^{-1}$, is defined such that $y \in F(x)$ implies $x \in F^{-1}(y)$. This is expressed as $ F^{-1}(y) = \{x \in X \mid y \in F(x)\}.$

In modern variational analysis, examining the regularity of set-valued mappings is essential for interpreting various mathematical models, especially in fields such as optimization, control theory, and economics. The regularity of these mappings refers to the characteristics that determine the local behaviour of the mapping around a point in its domain. Here, we focus solely on properties that are relevant to our research.
\begin{definition}\label{defRegularity}
	Let $F : X \rightrightarrows Y$ be a set-valued mapping and let $(\bar{x}, \bar{y}) \in \mathrm{gph}\, F$
	be a given point. We say that $F$ is:
	\begin{enumerate}
		\item[\rm (i)]  \textit{metrically subregular} at $(\bar{x}, \bar{y})$ if there exists $\kappa \geq 0$ along with some \nn $U$ of $\bar{x}$ such that
		\begin{eqnarray*}
			\label{eqSubreg}
			\mathrm{dist}(x, F^{-1}(\bar{y})) \leq \kappa\, \mathrm{dist}(\bar{y}, F(x)) \quad  \text{for each}\quad x \in U;
		\end{eqnarray*}
		\item[\rm (ii)] \textit{strongly metrically subregular} at $(\bar{x}, \bar{y})$ if there exists $\kappa\geq0$ and a \nn $U$ of $\bar{x}$ such that
		\begin{eqnarray*}
		\Vert x-\bx\Vert \leq \kappa \dist(\by, F(x))\quad\text{for each} \quad x\in U;
		\end{eqnarray*}
		\item[\rm (iii)] (strongly) \textit{metrically subregular around} $(\bar{x}, \bar{y}) \in \mathrm{gph}\, F$ if there is a \nn $W$ of $(\bar{x}, \bar{y})$ and $\kappa\geq0$ such that  at each $(x, y) \in \mathrm{gph}\, F \cap W$ there is a \nn  $U$ of $x$ such that
		\begin{eqnarray*}
			\Vert  u-x\Vert\leq \kappa \dist(y, F(u))\quad\text{for each}\quad u\in U.
	\end{eqnarray*}
	\end{enumerate}
\end{definition}
Strong metric subregularity around the reference point was first introduced in \cite{GO2022}; also see \cite{DR2014, Ioffe2017} for the additional properties.
\begin{definition}	Let $F: X \rightrightarrows Y$ be a mapping and let $(\bar{x}, \bar{y}) \in \gph\,F$. We say that
	\begin{enumerate}[(i)]
		\item $F$ is \emph{calm} at $(\bar{x}, \bar{y})$ if there exist $\mu \geq 0$ and a \nn $U\times$ of $(\bar{x},\by)$ such that
		\[
		F(x)\cap V \subset F(\bx) + \mu \Vert x-\bar{x}\Vert\ball_{Y}
		\quad\text{for all } x \in U;
		\]
		\item $F$ is \emph{isolated calm} at $(\bar{x}, \bar{y})$ if there exist $\mu\ge 0$ and a \nn $U\times v$ of $(\bar{x},\by)$ such that
		\[
		F(x)\cap\, V\subset \{\bar{y}\} + \mu \Vert x-\bar{x}\Vert \ball_{Y}
		\quad\text{for all}\quad x \in U.
		\]
	\end{enumerate}
\end{definition}
\begin{remark}
Note that a single-valued mapping $f:X\longrightarrow Y$ is (isolated) calm at $\bx$ if there exist $\mu\ge0$ and a \nn $U$ of $\bx$ such that
	\begin{eqnarray*}
		\Vert f(x)-f(\bx)\Vert\leq \mu \Vert x-\bx \Vert\quad\text{for each}\quad x\in U.
	\end{eqnarray*}
\end{remark}
It is known that strong metric subregularity is stable with calm single-valued  perturbation \cite{CDK2018}. 
\begin{theorem}
	Let $\alpha, \kappa, \mu$ be positive constants such that $\kappa\mu < 1$.
	Consider a mapping $F : X \rightrightarrows Y$ which is strongly subregular at $(\bar{x}, \bar{y})$ 
	with the constant $\kappa$ and a \nn $\ball_X[\bar{x},\alpha]$, and a function 
	$g : X \longrightarrow Y$ which is calm at $\bar{x}$ with constant $\mu$ and a \nn $\ball_X [\bar{x},\beta]$.
	Then $g + F$ is strongly metrically subregular at $(\bar{x}, \bar{y} + g(\bar{x}))$ 
	with the constant $\kappa/(1-\kappa \mu)$ and the \nn $\ball_X [\bar{x}, \alpha]$.
\end{theorem}
It should be noted that metric subregularity does not necessarily hold when a calm single-valued perturbation is applied. The example below illustrates this with a counterexample.
\begin{example1}
 Define two functions $f$ and $g$ such that 
	\[
	f(x)=
	\begin{cases}
		0, & x\le 0,\\[1mm]
		x, & x>0,
	\end{cases}
	\quad\text{and}\quad
		g(x)=
	\begin{cases}
		x^2, & x\le 0,\\[1mm]
		-x^2, & x>0,
	\end{cases}
	\]
	where $f$ is metrically subregular at \(0\)  and $g$ is calm at $0$.

	Since \(|g(x)|\le |x|\) near \(0\), \(g\) is (isolated) calm. Now, consider the function \(h:=f+g\):
	\[
	h(x)=
	\begin{cases}
		x^2, & x\le 0,\\[1mm]
		x-x^2, & x>0.
	\end{cases}
	\]
	Metric subregularity of \(h\) at \(0\) would require the existence of \(\kappa>0\) such that
	\[
	|x|\le \kappa\,|h(x)|
	\]
	in a \nn of \(0\). For \(x<0\), we have \(h(x)=x^2\), so
	\[
	|x|\le \kappa|x^2|\quad\Rightarrow\quad \frac{1}{|x|}\le \kappa.
	\]
	This inequality fails as the left-hand side goes to infinity when \(x\downarrow 0\).
\end{example1}
Using ideas from \cite{CDK2018}, the following proposition establishes the stability of strong metric subregularity at the reference point under set-valued perturbations.
\begin{proposition}
	\label{propStability}
	 Consider a set-valued mappings $F:X\tto Y$ and a point $(\bx, \by)\in \gph\,F$. Assume that there are $\kappa>0$ and  $\alpha>0$ such that $F$ is strongly metrically subregular at $(\bx,\by)$ with the constant $\kappa$ and  the \nn $\ball_{X}[\bx,\alpha]$. Let $\mu>0$ be such that $\kappa \mu <1$ and let $\kappa'>\kappa/(1-\kappa \mu)$.
	
	Then for each $\beta\in (0,\alpha]$	and for each set-valued mapping $G:X\tto Y$ satisfying
	\begin{eqnarray}
		\label{eqStrongerCalmness}
		G(\bx)=\lbrace \bz \rbrace\quad \text{and}\quad	G(x)\subset \lbrace\bz\rbrace+\mu  \Vert x-\bx \Vert\ball_{Y}\quad\text{for each}\quad x \in \ball_X[\bx,\beta],
	\end{eqnarray}
	we have that the mapping $G+F$ is strongly metrically subregular at $(\bx,\by+\bz)$ with the constant $\kappa'$ and  the \nn $\ball_{X}[\bx,\beta]$. 
\end{proposition}
\begin{proof}
	Fix any $\mu>0$, $\beta>0$, and any mapping $G$ as in the conclusion.  Fix any  $x\in \ball_X[\bx, \beta]$.
	
	If $(G(x)+F(x)) =\emptyset$, we are done.  If not, fix any $z\in G(x)$, then
	$$
	\Vert z-\bz\Vert \leq \mu \Vert x-\bx\Vert.
	$$
	Then
	\begin{eqnarray*}
		\Vert x-\bx\Vert&\leq& \kappa \dist(\by,F(x))\leq \kappa \dist(\by+\bz,z+F(x))+\kappa\Vert z-\bz\Vert\\
		& \leq &\kappa \dist(\by+\bz,z+F(x))+\kappa\mu \Vert x-\bx\Vert.
	\end{eqnarray*}
	Taking into account that $\kappa\mu<1$ and $\tfrac{\kappa}{1-\kappa\mu}<\kappa'$ and that $z$ is fixed arbitrary, we obtain
	\begin{eqnarray*}
		\Vert x-\bx\Vert&\leq&\tfrac{\kappa}{1-\kappa \mu} \dist(\by+\bz,G(x)+F(x))\leq\kappa' \dist(\by+\bz,G(x)+F(x)).
	\end{eqnarray*}
	
\end{proof}
Note that the property in \eqref{eqStrongerCalmness} is stronger than isolated calmness at $(\bx,\bz)$. Moreover, since $G(\bx)$ is a singleton and $G$ is upper semicontinuous at $\bx$, it follows that $F$ and $G$ are sum-stable in the sense of \cite{NTT2014}.

\section{Uniformity of strong subregularity}

We are  investigating uniformity of strong metric subregularity around the reference point on compact subsets of Banach spaces of mappings which are defined as a sum of a single-valued (possibly nonsmooth) mapping and a set-valued mapping. We are following ideas of the proofs from \cite[Section 2]{CPR2019}.

First, we present a statement concerning perturbed strong metric subregularity on a set.

\begin{theorem}
	\label{thmStabilitySubregularity}
	 Consider a set-valued mappings $F:X\tto Y$ and a point $(\bx, \by)\in \gph\,F$. Assume that there are positive constants $a,b,$ and $\kappa$ such that  for each $(x,y)\in \big(\ball_X[\bx, a]\times \ball_Y[\by, b])\cap \gph\,F$ there is $r>0$ such that for each $u\in \ball_X[x, r]$ we have
	\begin{eqnarray}
		\label{eqConstrainedSubreg}
		\Vert u-x\Vert \leq \kappa \dist(y,F(u)\cap \ball_Y[\by, b]).
	\end{eqnarray}
	Let $\mu>0$ be such that $\kappa \mu <1$ and let $\kappa'>\kappa/(1-\kappa \mu)$. Then for every positive $\alpha$ and $\beta$ such that
$		2\alpha\leq a$ and $	2\beta+\mu \alpha \leq b$	and for every mapping $g:X\longrightarrow Y$ satisfying
	\begin{eqnarray*}
	\Vert g(\bx)\Vert\leq\beta\quad\text{and}\quad	\Vert g(x)-g(u)\Vert \leq \mu\Vert x-u\Vert\quad\text{for each}\quad x,u \in \ball_X[\bx,\alpha],
	\end{eqnarray*}
	 we have that for each $y\in \ball_Y[\by,\beta]$ and each $x\in (g+F)^{-1}(y)\cap \ball_X[\bx, \alpha]$  there is $r\in (0,\alpha]$  such that each $u\in\ball_X[x,r]$ and   each   $ v\in  (g+F)(u)\cap \ball_Y[\by, \beta]$   we have
	 \begin{eqnarray*}
	 	\Vert u-x\Vert\leq \kappa' \Vert y-v\Vert.
	 \end{eqnarray*}
	\end{theorem}
	\begin{proof}
		Fix any $\alpha>0$ and $\beta>0$ and any mapping $g$ as in the conclusion. Then fix any $y\in \ball_Y[\by,\beta]$. Fix any $x\in (g+F)^{-1}(y)\cap \ball_X[\bx, \alpha]$ and find a corresponding $r\in (0, \alpha]$ such that \eqref{eqConstrainedSubreg}, for each $u\in \ball_X[x, r]$, holds.
Fix any $u\in \ball_X[x,r]$.
Then $u\in \ball_X[\bx, a]$ since
\begin{eqnarray*}
	\Vert u-\bx\Vert \leq\Vert u-x\Vert+\Vert x-\bx\Vert\leq r+{\alpha}\leq 2\alpha\leq a.
\end{eqnarray*}
Note that $ \ball_Y[\by,\beta]\subset \ball_Y[\by+g(x), b]$, since
for each $v \in  \ball_Y[\by,\beta]$ we have
\begin{eqnarray*}
	\Vert\by+ g(x)- v\Vert\leq 	\Vert\by-  v\Vert +\Vert g(x)-g(\bx)\Vert +\Vert g(\bx)\Vert\leq 	\beta+\mu r +\beta \leq 	2\beta+\mu \alpha  \leq b;
\end{eqnarray*}
so is $y-g(x)\in \ball_{Y}[\by, b]$. If $(g(u)+F(u))\cap \ball_Y[\by, \beta] =\emptyset$, we are done. If not, fix any $v\in (g(u)+F(u))\cap \ball_Y[\by, \beta]$. Then
\begin{eqnarray*}
	\Vert u-x\Vert&\leq& \kappa \dist(y-g(x),F(u)\cap \ball_Y[\by, b])\leq \kappa \dist(y-g(u),F(u)\cap \ball_Y[\by, b])+\kappa\Vert g(u)-g(x)\Vert\\
	&\leq &  \kappa \dist(y,(g(u)+F(u))\cap \ball_Y[\by+g(u), b])+ \kappa\mu \Vert u -x\Vert\\&\leq&  \kappa \dist(y,(g(u)+F(u))\cap \ball_Y[\by, \beta])+ \kappa\mu  \Vert u-x\Vert\\
	&\leq&  \kappa \Vert y -v \Vert+ \kappa\mu \Vert u-x\Vert.
\end{eqnarray*}
Taking into account that $\kappa\mu<1$ and $\tfrac{\kappa}{1-\kappa\mu}<\kappa'$, we obtain
\begin{eqnarray*}
	\Vert u-x\Vert\leq \tfrac{\kappa}{1-\kappa\mu} \Vert y -v\Vert\leq  \kappa' \Vert y-v\Vert.
\end{eqnarray*}
	\end{proof}

We will now demonstrate that subregularity around each point of a compact set implies uniform subregularity. In other words, it is possible to find the same constant and \nn for all points within this set.

	\begin{theorem}
		\label{thmParametricStabilitySubregularity}
		Let  $\Omega \subset P\times X$ be a compact set. Consider a set-valued mapping $F:X\tto Y$ and a continuous single-valued mapping $f: P\times X\longrightarrow Y$ such that for each $({t},\bx)\in \Omega$ we have:
		\begin{enumerate}
			\item[\rm (i)]  the mapping $X\ni x \longmapsto G_{{t}}(x):= f({t},x)+F(x)$  is strongly metrically subregular around $(\bx,0)$;
			\item[\rm (ii)]  for each $\mu>0$ there is $\alpha>0$ such that for each $x, u \in \ball_X[\bx,\alpha]$ and each $s \in\ball_P[{t},\alpha]$ we have
			\begin{eqnarray*}
				\Vert f(s,u)-f({t},u)-(f(s,x)-f({t},x))\Vert \leq \mu \Vert x-u\Vert.
			\end{eqnarray*}
		\end{enumerate}
		Then:
		\begin{enumerate}
					 \item[\rm (iii)] there are positive constants $\kappa$ and $a,  b$  such that for each $({t},\bx)\in \Omega$ the mapping $G_{{t}}$ is strongly metrically subregular around $(\bx,0)$ with the constant $\kappa$ and \nns\,$\ball_X[\bx, a]$ and $\ball_Y[0, b]$;
					 \item[\rm (iv)] there are $\kappa>0$ and $a>0$ such that for each $(t,x)\in \Omega$ the mapping $G_t$ is strongly metrically subregular at $(x,0)$ with the constant $\kappa'$ and the \nn $\ball_{X}[x,a]$.
		\end{enumerate}
	\end{theorem}
	\begin{proof} We are showing only {\rm (iii)}. The proof of (iv) follows similarly from Proposition \ref{propStability}, see also \cite[Theorem 2.6]{CPR2019}.
Fix any $({t},\bx) \in \Omega$. Find positive $a, b$, and $\kappa$, such that for each $(x,y)\in \big(\ball_{X}[\bx,a]\times \ball_{Y}[0, b]\big)\cap\gph\,G_{{t}}$ there is $r>0$ such that for each $u\in \ball_X[x,r]$ we have
\begin{eqnarray*}
	\Vert u-x\Vert\leq \kappa \dist(y,G_{{t}}(u)). 
\end{eqnarray*}
Let $\mu:=1/(2\kappa)$ and $\kappa':= 3\kappa$. Then $\kappa \mu <1$ and $\kappa'>2\kappa=\kappa/(1-\kappa \mu)$. Find $\alpha\in \left(0,  b/(2\mu)\right)$ such that each $x, u \in \ball_X[\bx,2\alpha]$ and each $s \in\ball_P[{t},\alpha]$ we have
\begin{eqnarray*}
	\Vert f(s,u)-f({t},u)-(f(s,x)-f({t},x))\Vert \leq \mu \Vert x-u\Vert.
\end{eqnarray*}
Let $\beta:=b/4$. Then $2\beta +\mu \alpha<b/2+b/2=b.$ Since $f$ is continuous, there is $r'\in (0,\alpha/2]$ such that
\begin{eqnarray*}
	\Vert f(s,\bx)-f({t},\bx)\Vert\leq \beta\quad \text{for each}\quad s \in\ball_P[{t},r'].
\end{eqnarray*}
 Fix any $(s, x) \in\big(\ball_P[{t}, r']\times \ball_{X}[\bx, r']\big)\cap \Omega$.  Define a mapping $g: X\longrightarrow Y$ such that
\begin{eqnarray*}
	g(u):= f(s,u)-f({t},u)\quad\text{for}\quad u\in X.
\end{eqnarray*}
Then $G_{s}=g+G_{{t}}$ and for  each $x, u \in \ball_X[\bx,2\alpha]$ we have
	\begin{eqnarray*}
		\Vert g(x)-g(u)\Vert \leq \mu \Vert x-u\Vert\quad\text{and}\quad \Vert g(\bx)\Vert \leq \beta.
	\end{eqnarray*}
Theorem \ref{thmStabilitySubregularity}, with $G:=G_{{t}}$ and $\by:=0$, implies that for each $y\in \ball_Y[0,\beta]$ and each $x\in (g+G_{{t}})^{-1}(y)\cap \ball_X[\bx, \alpha]$  there is $r\in (0,\alpha]$ such that for each $u\in\ball_X[x,r]$ and  each   $ v\in  (g+G_{{t}})(u)\cap \ball_X[0, \beta]$   we have
\begin{eqnarray*}
	\Vert u-x\Vert\leq \kappa' \Vert y-v\Vert.
\end{eqnarray*}
We are showing that for each $(x,y) \in\big(\ball_{X}[\bx,\alpha/3]\times\ball_{Y}[0,\beta/3]\big)\cap \gph\,G_s$ there is $r>0$ such that for each $u\in \ball_{X}[x, r]$ we have
\begin{eqnarray*}
	\Vert u-x\Vert\leq \kappa'\dist(y,G_s(u)).
\end{eqnarray*}
Fix any such $(x, y)$ and find a corresponding $r\in (0,2\kappa'\beta/3]$ as in the claim and fix any $u\in \ball_{X}[x, r]$.  Thus $x\in (g+G_t)^{-1}(y)\cap \ball_{X}[\bx,\alpha]$. Fix any $v\in G_{s}(u)$. If $\Vert v\Vert\leq \beta$, using the claim, we get $\Vert u-x\Vert\leq \kappa' \Vert y-v\Vert$.

 If $\Vert v\Vert> \beta$, then $\Vert y-v\Vert\geq \Vert v\Vert-\Vert y\Vert> \beta-\beta/3=2/3\beta$  and so
 \begin{eqnarray*}
 	\Vert u - x \Vert\leq r\leq 2\kappa'\beta/3< \kappa'\Vert y-v\Vert.
 \end{eqnarray*}
 To sum up, we show that for each $({t},\bx)\in \Omega$ there are constants $\kappa'>0$, $\alpha>0,$ $\beta>0$, and $r'\in (0,\alpha/2)$ such that for each $(s, u) \in\big(\ball_P[{t}, r']\times \ball_{X}[\bx, r']\big)\cap \Omega$  and each $(x,y) \in\big(\ball_{X}[u,\alpha]\times\ball_{Y}[0,\beta]\big)\cap \gph\,G_{{s}}$ there is $r>0$ such that for each $v\in \ball_{X}[x, r]$ we have
 \begin{eqnarray*}
 \Vert v-x\Vert \leq \kappa'\dist(y,G_s(v)).
 \end{eqnarray*}
 Note that  $\ball_{X}[u,r']\subset\ball_{X}[\bx,\alpha]$, then $G_s$ is strongly metrically subregular around $(x,0)$  with the constant $\kappa'$ and \nns $\ball_{X}[x,\alpha]$ and $\ball_{Y}[0, \beta]$.
 
 So $\kappa'$, $\alpha,$ $\beta$, and $r'$ depends only on the choice $(\bar{t},\bx)\in \Omega$. Then from open covering $\cup_{z=(t,x)\in\Omega} \big(\ball_P(t,r_z')\times\ball_{X}(x, r_z'))$ of compact set $\Omega$ find a finite subcovering $\mathcal{O}_i:= \big(\ball_P(t_i,r_{i}')\times\ball_{X}(x_i, r_{i}'))$ for $i=1,2,3, \dots, N$.
 Let $a:=\min\lbrace \alpha_{i}: i=1,2,3, \dots, N\rbrace $, $b:=\min\lbrace \beta_i: i=1,2,3, \dots, N\rbrace $, and $\kappa:=\max\lbrace \kappa'_i: i=1,2,3, \dots, N\rbrace.$ For any $(t,x)\in \Omega$ there is an index $i\in \lbrace 1,2,3,\dots, N\rbrace$ such that $(t,x)\in\mathcal{O}_i$. Hence the mapping $G_t$ is strongly metrically subregular around $(x,0)$ with the constant $\kappa$ and \nns $\ball_{X}[x,a]$ and $\ball_{Y}[0, b]$.
\end{proof}

Note that if $f$ is continuously differentiable, then condition {\rm (ii)} is satisfied.

Similarly to the previous result, strong metric subregularity at each point of a compact set implies  strong metric subregularity on the entire set

\begin{theorem}
\label{thmParametricStabilitySubregularityAt}
Let  $\Omega \subset P\times X$ be a compact set. Consider a set-valued mapping $F:X\tto Y$ and a continuous single-valued mapping $f: P\times X\longrightarrow Y$ such that for each $({t},\bx)\in \Omega$ we have:
\begin{enumerate}
	\item[\rm (i)]  the mapping $X\ni x \longmapsto G_{{t}}(x):= f({t},x)+F(x)$  is strongly metrically subregular \textbf{at} $(\bx,0)$;
	\item[\rm (ii)]  for each $\mu>0$ there is $\alpha>0$ such that for each $x, u \in \ball_X[\bx,\alpha]$ and each $s \in\ball_P[{t},\alpha]$ we have
	\begin{eqnarray*}
		\Vert f(s,u)-f({t},u)-(f(s,x)-f({t},x))\Vert \leq \mu \Vert x-u\Vert.
	\end{eqnarray*}
\end{enumerate}
Then there are $\kappa>0$ and $c>0$ such that for each $(t,x)\in \Omega$ the mapping $G_t$ is strongly metrically subregular \textbf{at} $(x,0)$ with the constant $\kappa$ and the \nn $\ball_{X}[x,c]$.
\end{theorem}
\begin{proof}
	The proof is similar to the proof of Theorem \ref{thmParametricStabilitySubregularity}, but instead of using Theorem \ref{thmStabilitySubregularity}, we apply Proposition \ref{propStability}.
\end{proof}

The following statement guarantees that uniform strong metric regularity is preserved along continuous solution trajectories for the widely studied parametric generalized equation, for $T>0$, given by
$$
p(t)\in f(t,x(t))+F(x(t))\quad \text{for each}\quad t\in [0,T],
$$
where $p:[0,T]\longrightarrow Y, x:[0,T]\longrightarrow X$,  $f:[0,T]\times X\longrightarrow Y$, and $F:X\tto Y$.

 This means that as one follows a continuous path within the domain, the property of strong metric subregularity remains consistent and uniform, providing stability in the behaviour of the system

	\begin{theorem}
		\label{thmStabilitySubreg}
		Let $T>0$ be given. Consider a set-valued mapping $F:X\tto Y$ and a continuous single-valued mapping $f:[0,T]\times X\longrightarrow Y$, and two continuous mappings $x:[0,T]\longrightarrow X$ and $p:[0,T]\longrightarrow Y$ such that 
		\begin{enumerate}
			\item[\rm (i)]  for each $t\in [0,T]$ the mapping $X\ni x\longmapsto G_t(x):=f(t,x)+F(x)$ is strongly metrically subregular around $(x(t),p(t))$;
			\item[\rm (ii)]  for each $t\in [0,T]$ and each $\mu>0$ there is $\delta>0$ such that for each $x,u\in \ball_{X}[x(t),\delta]$ and each $s\in (t-\delta, t+\delta)$ we have
			\begin{eqnarray*}
				\Vert f(s,u)-f(t,u)-(f(s,x)-f(t,x))\Vert \leq \mu \Vert x-u\Vert.
			\end{eqnarray*}
		\end{enumerate}
		Then: \begin{enumerate}
			\item[\rm (iii)] there are positive constants $a, b$, and $\kappa$ such that for each $t\in [0,T]$ the mapping $G_t$ is strongly metrically subregular around $(x(t),p(t))$ with the constant $\kappa$ and \nns $\ball_X[x(t),a]$ and $\ball_{Y}[p(t),b]$;
			\item[\rm (iv)] there are $c>0$ and $\kappa'>0$ such that for each $t\in [0,T]$ the mapping $G_t$ is strongly metrically subregular at $(x(t),p(t))$ with the constant $\kappa'$ and \nn $\ball_X[x(t), c]$.
			\end{enumerate}
	\end{theorem}
	\begin{proof}
For {\rm (iii)}, apply Theorem \ref{thmParametricStabilitySubregularity}, and for {\rm (iv)}, apply Theorem \ref{thmParametricStabilitySubregularityAt}, both with $P := [0, T] \times Y$, the compact set $\Omega := \bigcup_{t \in [0, T]} (t, p(t), x(t))$, and the function $f(t, x) := f(p, x) - y$ for $t = (p, y) \in P$ and $x \in X$. 
	\end{proof}

\section{Conclusion}

In this paper, we have provided a comprehensive study of strong metric subregularity and its uniformity in Banach spaces. Our investigation has led to the following key contributions. We extended the local concept of strong metric subregularity to a uniform version over compact sets, thereby enabling the use of a common constant and \nn for all points in a given compact set.

These findings not only reinforce the theoretical underpinnings of variational analysis but also offer new avenues for the development of robust computational schemes in the context of generalized equations. Future work may explore further extensions of these uniformity results to broader classes of perturbations and more general frameworks, thereby enhancing both the theory and practical applications in optimization and control.


\end{document}